\documentclass[11pt,reqno]{amsart}

\newtheorem{lemma}{Lemma}
\newtheorem{definition}[lemma]{Definition}
\newtheorem{proposition}[lemma]{Proposition}
\newtheorem{theorem}[lemma]{Theorem}
\newtheorem{remark}[lemma]{Remark}

\newtheorem{example}[lemma]{Example}

\usepackage{amssymb,amsmath,amsthm}
\usepackage{epsfig,amssymb,amsmath,amsthm,graphicx,psfrag}
\usepackage[all]{xy}
\usepackage{enumerate}
\usepackage{color}

\begin{document}
\title{Periodic String Complexes over String Algebras}

 \author{Andr\'es Franco*, Hern\'an Giraldo*, and Pedro Rizzo*.}

\date{\today}

\begin{abstract}
In this paper we develop combinatorial techniques for the case of string algebras with the aim to give a characterization of string complexes with infinite minimal projective resolution. These complexes will be called \textit{periodic string complexes}. As a consequence of this characterization, we give two important applications. The first one, is a sufficient condition for a string algebra to have infinite global dimension. In the second one, we exhibit a class of indecomposable objects in the derived category for a special case of string algebras. Every construction, concept and consequence in this paper is followed by some illustrative examples.
\end{abstract}
\maketitle


\vskip5pt
\noindent
 
\vskip5pt
\noindent
 * Instituto de Matem\'aticas, Universidad de Antioquia, Medell\'in-Colombia

 e-mail: andres.francol@udea.edu.co, \ hernan.giraldo@udea.edu.co, \ pedro.hernandez@udea.edu.co.

\section{Introduction}
Projective resolutions have played an important role in the development of algebra with many applications in representation theory, ring theory, homological theory and algebraic geometry (\cite{GSZ} and references therein). Many theoretical, combinatorial and computational techniques  have been introduced in order to obtain these resolutions (\cite{FGKK}, \cite{GHZ}). In this paper we develop combinatorial techniques for the case of string algebras with the aim to give a characterization of string complexes with infinite minimal projective resolution. These complexes will be called \textit{periodic string complexes}.

As a consequence of this characterization, we give two important applications. The first one, is a sufficient condition for a string algebra to have infinite global dimension. In the second one, we use the functor in \cite{Be-Me} to prove that string complexes are indecomposable objects in the derived category of a special case of string algebras. In this special case we say that the string algebra satisfies the \textit{unique maximal path property}. Every construction, concept and consequence in this paper is followed by some illustrative examples.

Actually, this work has deep connections with the ideas and consequences in \cite{FGR}. In fact, some results of this paper were announced without proof in \cite{FGR}. However, we consider these results could have independent interest, hence detailed proofs are presented here.

In general, it is not easy to describe the indecomposable objects in the derived category of an algebra. In \cite{Be-Me}, the authors solved this problem when the algebra is Gentle. A recent work in \cite{FGR}, using an adaptation of the constructions in \cite{Be-Me}, the authors prove that string and band complexes are indecomposable objects in the bounded derived category of a new class of algebras: the String Almost Gentle algebras.

On the other hand, P. Bergh, Y. Han and D. Madsen in \cite{Be-Han-Ma} (see also \cite{Ha-Za}) gave a sufficient condition for a monomial algebra to have infinite global dimension. In this paper we will give an analogous sufficient condition for the case of a string algebra using different combinatorial techniques. 

This paper is organized as follows. In section \ref{Sec-Pre}, we fix some notations and present some background material about well-known results in derived categories, string algebras and string complexes. 

Then, in section \ref{Sec:PSC} we state and prove the main Theorem. We introduce a class of complexes characterized by having infinite minimal projective resolution. These complexes will be called \textit{periodic string complexes}. The main technique is the generalization of the combinatorial sets $Q_c$, $GSt_c$, $GSt^c$ introduced in  \cite{Be-Me} and the proof follows some ideas in \cite{FGR}.

Finally, in section \ref{Sec:App} we introduce and prove two important applications of our main result, as we mentioned above.

\section{Preliminaries}\label{Sec-Pre}

\subsection{About derived categories of modules over finite dimensional algebras}
Let $A$ be a finite-dimensional algebra of the form $kQ/I=k(Q, I)$ over an algebraically closed field $k$, where $I$ is an admissible ideal of $kQ$ and $Q$ is a finite quiver. As usual, we denote by $Q_0$ (resp. by $Q_1$) the set of vertices (resp. the set of arrows) of $Q$. Let $A-\text{mod}$ be the category of finitely generated left $A$--modules.

We denote by $e_i$ the trivial path at vertex $i\in Q_0$ and by $P_i=Ae_i$ the corresponding indecomposable projective $A$--module.  We will denote by $\textbf{Pa}$ the set of all paths of $k(Q,I)$, that is, the set of all paths of $Q$ that are outside $I$ and by $\textbf{Pa}_{\geq l}$ (resp. by $\textbf{Pa}_{> l}$) the set of all paths in $\textbf{Pa}$ of length greater than or equal to a fixed non-negative integer $l$ (resp. greater than $l$). Every element of $A=k(Q,I)$ can be represented uniquely by a linear combination of elements in $\textbf{Pa}$, and hence we can assume that $\textbf{Pa}$ is a basis for $A$.

The set $\textbf{M}$ of \textit{maximal paths} in $\textbf{Pa}$ has a very important role throughout this theory. A path $w=w_1\cdots w_n$ in  $k(Q,I)$ is \textit{maximal} in $k(Q,I)$ if for all arrows $a,b \in Q_1$, we have that  $aw$ and $wb$ are not paths in $k(Q,I)$. Furthermore, a nontrivial path $w$ of $Q$ is in $\textbf{Pa}$ if and only if it is a sub-path of a maximal path  $\widetilde{w}$ that is not in $I$ (that is, of an element of $\textbf{M}$). This maximal path has the form $\widetilde{w}=\hat{w}w\bar{w}$ with $\hat{w}, \bar{w}\in \textbf{Pa}$. For future reference, the path $\hat{w}\in\textbf{Pa}$ is called a \emph{left completion of $w$}.

We denote by $D(A)$ (resp. $D^-(A)$ or $D^b(A)$) the derived category of $A-\text{mod}$ (resp. the derived category of right bounded complexes of $A-\text{mod}$ or the derived category of bounded complexes of $A-\text{mod}$). We denote by $C^b(\text{pro} \ A)$ (resp. $C^-(\text{pro} \ A)$ or $C^{-,b}(\text{pro} \ A)$) the category of bounded projective complexes (resp. of right bounded projective complexes or right bounded projective complexes with bounded cohomology).

Also, we denote by $K^b(\text{pro} \ A)$ (resp. $K^-(\text{pro} \ A)$ or $K^{-,b}(\text{pro} \ A)$) the corresponding homotopy categories to $C^b(\text{pro} \ A)$ (resp. $C^-(\text{pro} \ A)$ or $C^{-,b}(\text{pro} \ A)$).

By $\mathfrak{p}(A)$ we denote the full subcategory of $C^b(\text{pro} \ A)$ defined by the projective complexes such that the image of every differential map is contained in the radical of the corresponding projective module. It is well known that $\mathfrak{p}(A)$ is a Krull-Schmidt category.

Any projective complex is the sum of two complexes: one complex in $\mathfrak{p}(A)$ and another one isomorphic to the zero object in the derived category (because all differential maps are 0's or isomorphisms). Hence, we can assume that all the complexes we deal with are in $\mathfrak{p}(A)$.

Now, since every projective module is a finite direct sum of indecomposable projective modules, each morphism (differential) in the complex $P^{\bullet}$ is given by a block matrix of size $\sum_id_{i,j}\times\sum_id_{i,j+1}$ if the differential goes from the place $j$ to the place $j+1$. Here $d_{i,j}$ represents how many times the indecomposable projective module $P_i$ appears as a direct summand in each $P^j$ in the complex $P^{\bullet}$ and we will use the notation $P_i^{d_{i,j}}$. Thus, each block gives the component of the morphism corresponding to each pair of indecomposables. That is, each block corresponds to a morphism $P_r^{d_{r,j}}\longrightarrow P_s^{d_{s,j+1}}$.

It is well known that the paths $w\in\textbf{Pa}$ such that $s(w)=r$ and $t(w)=s$ form a basis for $\text{Hom}(P_r,P_s)$. In the particular case of the category $\mathfrak{p}(A)$ we can assume that only the paths $w\in\textbf{Pa}_{\geq 1}$ are involved, because trivial paths give isomorphisms.

If $w\in\textbf{Pa}_{\geq 1}$ is one of these paths, it defines the morphism $p(w): P_r\longrightarrow P_s$, given by the multiplication times $w$ on the right: $u\mapsto v=uw$. Thus, any homomorphism from $P_r$ to $P_s$ is associated to a linear combination of paths like  $w$.

An important and useful result, which we will use later, is that $D^b(A)$ is equivalent to $K^{-,b}(\text{pro} \ A)$ (see, \cite{KoZi}, Prop 6.3.1, p.113). In fact, this result allows us to classify the indecomposable objects in $D^b(A)$. Precisely,

\begin{proposition}\label{prop:1}
There exist spectroids $\text{ind} \ D^b(A)$ and $\text{ind} \ \mathfrak{p}(A)$  of $D^b(A)$ and $\mathfrak{p}(A)$, respectively, such that the set of objects of the spectroid satisfies
$$
\text{ind}_0 \ D^b(A)= \text{ind}_0 \ \mathfrak{p}(A) \cup \{\beta(M^\bullet)^\bullet \ | \ M^\bullet \in \mathcal{X}(A) \}.
$$
\end{proposition}
Here, $\mathcal{X}(A)$ is a fixed set of representatives of the quotient of the set 
$$
\overline{\mathcal{X}(A)}:=\{M^\bullet\in \text{ind}_0 \ \mathfrak{p}(A) \ | \ P_{\beta(M^\bullet)^\bullet}^\bullet \notin K^b(\text{pro} \ A) \}
$$ 
by the equivalence relation $\cong_{\mathcal{X}}$. This equivalence relation is defined on $\overline{\mathcal{X}(A)}$ in $K^{-,b}(\text{pro} \ A)$, by
$$
M^\bullet\cong_{\mathcal{X}}N^\bullet \quad \text{if and only if} \quad 
P_{\beta(M^\bullet)^\bullet}^\bullet\cong P_{\beta(N^\bullet)^\bullet}^\bullet
$$ 
where $P_{\beta(M^\bullet)^\bullet}^\bullet$ is the projective resolution of the good truncation $\beta(M^\bullet)^\bullet$ of the complex $M^{\bullet}$. Precisely, if $P^{\bullet}\in C^{b}(\text{pro} \ A)$, $P^{\bullet}\neq 0^{\bullet}$ and if $t$ is the greatest integer such that $P^i=0$ for $i<t$, then $\beta(P^{\bullet})^{\bullet}$, \emph{the good truncation of $P^{\bullet}$ below $t$}, is the complex given by
$$
\beta(P^{\bullet})^i=\left\{ \begin{array}{lcc}
                    P^i,       &   \text{if}  & i\geq t, \\
\ker \partial_{P^{\bullet}}^t, & \text{if}    & i=t-1,\\
                    0,         &              & \text{otherwise.}  
\end{array}
\right.
$$
$$
\partial_{\beta(P^{\bullet})^{\bullet}}^i=\left\{ \begin{array}{lcc}
      \partial_{P^{\bullet}}^i,       &   \text{if}  & i\geq t, \\
\iota_{\ker \partial_{P^{\bullet}}^t} & \text{if}    & i=t-1 \\
                           0,         &              & \text{otherwise,}  
\end{array}
\right.
$$
where $\iota_{\ker \partial_{P^{\bullet}}^t}$ is the inclusion map. More details in \cite{KoZi}, \cite{Be-Me}.

\begin{remark}\label{Rem:remark1}
If $A$ has finite global dimension, then $\mathcal{X}(A)=\emptyset$ and $\text{ind}_0 \ D^b(A)= \text{ind}_0 \ \mathfrak{p}(A)$.
\end{remark}

\subsection{About string algebras}
For the benefit of the reader we begin this section by recalling the definition of a special biserial algebra, and, in particular, of a string algebra.

\begin{definition}
 The algebra $A=kQ/I$ is called \emph{special biserial} if it satisfies the following conditions:
	\begin{enumerate}
		\item Any vertex of $Q$ is the starting point of at most two arrows. Any vertex of $Q$ is the ending point of at most two arrows.
		\item Given an arrow $a\in Q_1$, there is at most one arrow $b\in Q_1$ with $s(b)=t(a)$ and $ab\notin I$.
		\item Given an arrow $a\in Q_1$, there is at most one arrow $c\in Q_1$ with $t(c)=s(a)$ and $ca\notin I$.
	\end{enumerate}
\end{definition}

It is a well known fact that, for special biserial algebras, $I$  can be generated by zero relations and by commutativity relations. More details in \cite{SW}.

\begin{definition}
 A special biserial algebra $A=kQ/I$ is called a \emph{string algebra} if, additionally, $I$ is generated by zero relations, i.e., by paths of length greater than or equal to 2.
\end{definition}

The first preliminary result of our interest on string algebras, establishes the uniqueness of the left completion of an arrow $a$, denoted by $\hat{a}$.

\begin{lemma}\label{lem-left-completion}
 Let $A=kQ/I$ be a string algebra and let $a\in Q_1$. Then $\hat{a}$ is unique.
\end{lemma}

\begin{proof}
	Suppose there are two different left completions $\hat{a}$ and $\hat{b}$ of $a$. 
	
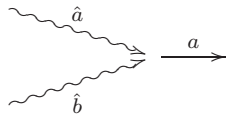
	$$
	\xymatrix@R=5mm{\ar@{~>}[rrd]^{\hat{a}}&&&\\
	                           &&\ar[r]^a&\\
                    \ar@{~>}[rru]_{\hat{b}}&&           &}
	$$
	Since $\hat{a}a$ and $\hat{b}a$ are elements in $\textbf{Pa}_{>0}$, then we have two compositions that not lie in $I$, which is a contradiction with the definition of a string algebra.
\end{proof}

Our next goal is to study the structure of the kernel of a morphism $p(w)$ for some $w\in\textbf{Pa}_{>0}$. From the definition of a string algebra, we know that there are at most two arrows $a$ and $b$ such that $t(a)=t(b)=s(w)$. We know also that it cannot happen that $aw\neq 0$ and $bw\neq 0$. Thus, at least one of $aw$ or $bw$ must be zero. Suppose, without loss of generality, that $bw=0$. For string algebras it can occur that $aw\neq 0$ but there could be a path $p$ of the form $p=p_1\cdots p_r a$, where $p_1,\dots, p_r\in Q_1$, such that $pw=0$.  Let us denote by $a^*$ the smallest path of this form (in case it exists). Then, $a^*$ is the smallest subpath of $\hat{a}a$ with the property $a^*w=0$. 

According to the above notation, we state the next Lemma, which will allow us to study the structure of $\ker p(w)$ for some $w\in\textbf{Pa}_{>0}$, and, consequently, the structure of the minimal projective resolutions of complexes. This result is a generalization of Lemma 5 in \cite{Be-Me} (also of Lemma 28 in \cite{FGR}).

\begin{lemma}[Lemma 43 \cite{FGR}]\label{Lem-St}
 Let $A=kQ/I$ be a string algebra and let $w\in\emph{\textbf{Pa}}_{>0}$. Then the general structure of the kernel of $p(w)$ is $\ker p(w)=Aa^*\oplus Ab$.
\end{lemma}

\begin{proof}[Proof]
Here we consider the general case in which $a^*w=0$ and $bw=0$ (we do not exclude the possibility that $a^*=a$). It is clear that $Aa^*\oplus Ab \subseteq \ker p(w)$. For the other inclusion, let $u$ be an element of $\ker p(w)$. Then, since $A$ is a string algebra $u$ must have either $a$ or $b$ as its last arrow. In the latter case $u\in Ab$. In the former case, by the minimality of $a^*$ we have that $a^*$ is a subpath of $u$ and hence $u\in Aa^*$ (here we are using the fact that both $u$ and $a^*$ are subpaths of $\hat{a}a$, which is unique as a consequence of Lemma \ref{lem-left-completion}).
\end{proof}

\begin{remark}
Two important examples of string algebras are gentle algebras \cite{Be-Me} and string almost gentle (briefly, SAG) algebras \cite{FGR}. Is it important to notice that there exist fundamental differences with these latter cases. The first one, is related with the smallest paths in Lemma \ref{Lem-St}, which are arrows themselves. Precisely, for Gentle algebras, if there are two arrows $a$ and $b$ ending at the vertex $s(w)$, then, necessarily, $bw=0$, $aw\neq 0$ and, so $b^*=b$. For SAG algebras it could happen that $aw=0$ and $bw=0$, and hence $a^*=a$ and $b^*=b$. The second one, is that for string algebras it could happen that different maximal paths have common arrows. However, as the following example (\ref{Ex1:2}) shows, there are string algebras, which are neither gentle nor SAG algebras, that satisfy the unique maximal path property. Recall that this property is automatically satisfied by all gentle and SAG algebras (see \cite{Be-Me} and \cite{FGR}, respectively)
\begin{example}\label{Ex:Str:1}
\begin{enumerate}
\item\label{Ex1:1} Consider the bound quiver $(Q, I)$ where 
$$
Q: \xymatrix{1\ar[r]^a&2\ar[r]^b&3\ar[r]^c&4}
$$ 
and $I=\langle abc\rangle$. Then $A=kQ/I$ is a string algebra which is neither, a gentle algebra nor a SAG algebra. In this case the set of maximal paths is $\emph{\textbf{M}}=\{ab, bc\}$. Thus, the arrow $b$ belongs to two different maximal paths, and so maximal paths are not unique for a given arrow.
\item\label{Ex1:2} Consider the bound quiver $(Q, I)$
$$ 
\xymatrix@C=4.5mm{ & 3\ar[dl]_c&                &&&6\ar[ld]_{c'}\\
		                              1\ar[rr]_a&& 2\ar[rr]_d\ar[ul]_b&&4\ar[rr]_{a'}&&5\ar[ul]_{b'}}
$$
where $I=\langle abc, ad, da', a'b'c'\rangle$. Then $A=kQ/I$ is a string algebra over $k$ that is neither, gentle nor SAG. The set of maximal paths is $\emph{\textbf{M}}=\{bcab, d, b'c'a'b'\}$. Thus, every arrow belongs to a unique maximal path.
\end{enumerate}
\end{example}
\end{remark}

\subsection{String complexes}
 The {\it string complexes} were introduced by Bekkert and Merklen in \cite{Be-Me}, which are associated to {\it generalized strings}. As we will see, some of these complexes in $D^b(A)$ constitute a class having infinite minimal projective resolution, when $A$ is a string algebra.

Let $A=kQ/I$ be a string algebra. Recall that, for each arrow $a\in Q_1$, we denote by $a^{-1}$  its \textit{formal inverse}, which verifies  $s(a^{-1})=t(a)$, $t(a^{-1})=s(a)$ and $(a^{-1})^{-1}=a$. Similarly, if $w=a_1\cdots a_n$ is a path in $Q$, the \textit{inverse path} of $w$ is given by $w^{-1}=a_n^{-1}\cdots a_1^{-1}$.  Then $s(a_i^{-1})=t(a_i)=s(a_{i+1})=t(a_{i+1}^{-1})$ and it is clear that $s(w^{-1})=t(w)$ and $t(w^{-1})=s(w)$.

Now, a \textit{walk} $\omega$ (resp. a \textit{generalized walk}) of length $n>0$ is a sequence $w_1\cdots w_n$, where each $w_i$ is either of the form $w$ or $w^{-1}$, $w$ being an arrow (resp. a path of positive length) and such that $t(w_i)=s(w_{i+1})$ for $i=1,\dots, n-1$. It is clear that $s(\omega)=s(w_1)$ and $t(\omega)=t(w_n)$. The notion of the \textit{inverse} of a walk (resp. of a generalized walk) is defined analogously as that for a path. Thus, the passage to inverses is an involutory transformation.

A \textit{closed walk} (resp. a \textit{closed generalized walk}) is a walk $\omega$ (resp. a generalized walk) such that $t(\omega)=s(\omega)$. In this case we consider its \textit{rotations} (or \textit{cyclic permutations}), denoted by $\omega[j]$, which are given by $\omega[j]=w_{j+1}\cdots w_nw_1\cdots w_j$, for $j=1,\dots,n-1$. 

The \textit{product} or \textit{concatenation} of two walks (resp. of two generalized walks) $\omega=w_1\cdots w_n$ and $\omega'=w_1'\cdots w_n'$ is defined as the walk (resp. generalized walk) $\omega\omega'=w_1\cdots w_nw_1'\cdots w_n'$, whenever $t(w_n)=s(w_1')$. 

We consider the following equivalence relation on the set of generalized walks, denoted by $\cong_s$: If $\upsilon$ and $\omega$ are two generalized walks, then 
$$
\upsilon\cong_s \omega\quad \text{if and only if}\quad \upsilon=\omega \ \text{or} \ \upsilon=\omega^{-1}.
$$

\begin{definition}
A \emph{string} is a walk $\omega=w_1\cdots w_n$ such that $w_{i+1}\neq w_i^{-1}$ for $1\leq i <n$ and such that no sub-word of $\omega$ or $\omega^{-1}$ is in $I$. The set of all strings in $(Q, I)$ will be denoted by $St$.
\end{definition}

\begin{example}
Consider the quiver $$Q: \ \xymatrix{1 \ar@<0.7ex>[r]^{a} \ar@<-0.5ex>[r]_c&2\ar@<0.7ex>[r]^{b} \ar@<-0.5ex>[r]_d &3}$$ with $I=\langle ad, cb\rangle$. We have that $\omega=abd^{-1}c^{-1}$ is a string but $\omega'=abb^{-1}c^{-1}$ and $\upsilon=abd^{-1}c^{-1}ad$ are not.
\end{example}

Now, we denote by $\overline{GSt}$ the set of all generalized walks $\omega=w_1\cdot w_2\cdots w_n$ satisfying

\begin{itemize}
	\item If $w_i, w_{i+1} \in \textbf{Pa}_{>0}$, then $w_iw_{i+1}\in I$.
	\item If $w_i^{-1}, w_{i+1}^{-1}\in \textbf{Pa}_{>0}$, then $w_{i+1}^{-1}w_i^{-1}\in I$.
	\item If $w_i, w_{i+1}^{-1} \in \textbf{Pa}_{>0}$ or $w_i^{-1}, w_{i+1} \in \textbf{Pa}_{>0}$, then $w_iw_{i+1}\in St$. 
\end{itemize}

We use the notation $GSt$ for a fixed set of representatives of the quotient of $\overline{GSt}$ over the equivalence relation $\cong_s$ together with all trivial paths. The elements of $GSt$ are called \textit{generalized strings}.

\begin{example}\label{Ex:Gst}
Consider the string algebra $A=kQ/I$ given by the quiver 
$$
Q: \ \xymatrix{1 \ar@<0.7ex>[r]^{a} \ar@<-0.5ex>[r]_c&2\ar@<0.7ex>[r]^{b} \ar@<-0.5ex>[r]_d &3}
$$ 
with $I=\langle ab, cd\rangle$. Then $\omega=a^{-1}\cdot c\cdot d\cdot (cb)^{-1}$ is a generalized string.
\end{example}

\begin{remark}
Notice that in example \ref{Ex:Gst} we use a different notation in order to distinguish  generalized strings from strings. Thus, we use a dot $\cdot$ between $w_i$ and $w_{i+1}$ in a generalized string. For instance, in the quiver  $Q: \ \xymatrix{2 \ar[r]^{a} & 1 & 3\ar[l]_{b}}$ we have that $ab^{-1}$ is a string and $a\cdot b^{-1}$ is a generalized string.
\end{remark}

Now, for a generalized walk $\omega=w_1\cdot w_2\cdot\cdots\cdot w_n$ we introduce the function $\mu_{\omega}: \{0,1,\ldots, n\}\longrightarrow\mathbb{Z}$, defined by 
$\mu_{\omega}(0):=0$ and 
$$\mu_{\omega}(i):=\left\{ \begin{array}{ll}
\mu_{\omega}(i-1)+1 ,  &  \text{if} \ \ w_i\in\textbf{Pa}_{>0}, \\
\mu_{\omega}(i-1)-1 ,  &  \text{if} \ \ w_i^{-1}\in\textbf{Pa}_{>0}.   
\end{array}
\right.$$

Thus, we set $\mu(\omega):=\min_{1\leq i\leq n}\{\mu_{\omega}(i)\}$. In example \ref{Ex:Gst}, for $\omega=a^{-1}\cdot c\cdot d\cdot (cb)^{-1}$, we have $\mu_{\omega}(0)=0, \mu_{\omega}(1)=-1, \mu_{\omega}(2)=0, \mu_{\omega}(3)=1, \mu_{\omega}(4)=0$, and hence $\mu(\omega)=-1$.

\begin{remark}
	In practice we will assume in general that $\mu_{\omega}(0)\leq \mu_{\omega}(n)$. This is possible because if $\mu_{\omega}(0)\geq\mu_{\omega}(n)$, then $\mu_{\omega^{-1}}(0)\leq\mu_{\omega^{-1}}(n)$ and $\omega^{-1}\cong_s \omega$. For instance, in the quiver of the previous example, for $\omega=b^{-1}\cdot a^{-1}$ we have that $\mu_{\omega}(0)=, \mu_{\omega}(1)=-1, \mu_{\omega}(2)=-2$ and for $\omega^{-1}=a\cdot b$, $\mu_{\omega^{-1}}(0)=0, \mu_{\omega^{-1}}(1)=1, \mu_{\omega^{-1}}(2)=2$.
\end{remark}

Now, for every generalized string, we will associate a finite projective complex called a \textit{string complex}.

\begin{definition}\label{def-st-comp}
Let  $\omega=w_1\cdot w_2\cdots w_n$ be a generalized string. Then  $P_{\omega}^\bullet$ is the projective complex $\xymatrix{\cdots P_{\omega}^i\ar[r]^{\partial_w^i}&P_{\omega}^{i+1}\cdots}$ defined as follows. The modules are given by
\begin{equation}\label{def-st-comp-mod}
P_{\omega}^i=\displaystyle\bigoplus_{j=0}^n\delta(\mu_{\omega}(j),i)P_{c(j)}
\end{equation}
where $c(0)=s(w_1)$, $c(j)=t(w_j)$ for $j>0$ and $\delta$ is the Kronecker delta. 
	
	The differential maps are given by $\partial_{\omega}^i=\left(\partial_{jk}^i\right)_{1\leq j,k\leq n}$, where
	\begin{equation}\label{def-string-complex-differentials}
	\partial_{jk}^i:=\left\{ \begin{array}{ll}
	p(w_{j+1}) ,  &  \text{if} \ \ w_{j+1}\in\emph{\textbf{Pa}}_{>0}, \mu_{\omega}(j)=i \ \text{and} \ k=j+1, \\
	p(w_j^{-1}),  &  \text{if} \ \ w_j^{-1}\in\emph{\textbf{Pa}}_{>0} , \mu_{\omega}(j)=i \ \text{and} \ k=j-1, \\
	0,                &  \text{otherwise}.   
	\end{array}
	\right.
	\end{equation}
	Also, for each trivial generalized string $e_i^{\pm 1}$, let us denote by $P_{e_i^{\pm} 1}^\bullet$, the following projective complex
	$$\xymatrix{\cdots\ar[r]&0\ar[r]&P^0=P_i\ar[r]^(0.7){\partial^0}&0\ar[r]&\cdots}$$
\end{definition}
\begin{example}
	Let $A=kQ/I$ be the string algebra given by the bound quiver $$Q: \ \xymatrix{1 \ar@<0.7ex>[r]^{a} \ar@<-0.5ex>[r]_c&2\ar@<0.7ex>[r]^{b} \ar@<-0.5ex>[r]_d &3}$$ with $I=\langle ab, cd\rangle$ and let $\omega$ be the generalized string $\omega=a^{-1}\cdot c\cdot d \cdot (cb)^{-1}=w_1\cdot w_2\cdot w_3\cdot w_4$. Then, as we have seen above, $\mu_{\omega}(0)=0, \mu_{\omega}(1)=-1, \mu_{\omega}(2)=0, \mu_{\omega}(3)=1, \mu_{\omega}(4)=0$. Thus, from \eqref{def-st-comp-mod} we have
	\begin{itemize}
		\item For $i=-1$, we have $P_{\omega}^{-1}=\displaystyle\bigoplus_{j=0}^4\delta(\mu_{\omega}(j),-1)P_{c(j)}=P_{c(1)}=P_{t(w_1)}=P_1$.
		\item For $i=0$, we have $P_{\omega}^{0}=\displaystyle\bigoplus_{j=0}^4\delta(\mu_{\omega}(j),0)P_{c(j)}=P_{c(0)}\oplus P_{c(2)}\oplus P_{c(4)}=P_2\oplus P_2\oplus P_1$.
		\item For $i=1$, we have $P_{\omega}^{1}=\displaystyle\bigoplus_{j=0}^4\delta(\mu_{\omega}(j),1)P_{c(j)}=P_{c(3)}=P_{t(w_3)}=P_3.$
	\end{itemize}
Now, let us calculate de differential maps. According to \eqref{def-string-complex-differentials},
$$\partial_{\omega}^{-1}=\left(\partial_{jk}^{-1}\right)_{0\leq j,k\leq 4}=\begin{pmatrix}
0                  & 0 & 0                 & 0  & 0\\
\boxed{p(a)}& 0 & \boxed{p(c)}& 0  & \boxed{0}\\
0                  & 0 & 0                & 0  & 0\\
0                  & 0 & 0                & 0  & 0\\
0                  & 0 & 0                & 0  & 0
\end{pmatrix}\longleftrightarrow\begin{pmatrix}
p(a) & p(c) & 0
\end{pmatrix}$$
and 
$$\partial_{\omega}^{0}=\left(\partial_{jk}^{0}\right)_{0\leq j,k\leq 4}=\begin{pmatrix}
0      & 0 & 0    & \boxed{0}       & 0\\
0      & 0 & 0    & 0                    & 0\\
0      & 0 & 0    & \boxed{p(d) } & 0\\
0      & 0 & 0    & 0                    & 0\\
0      & 0 & 0    & \boxed{p(cb)} & 0
\end{pmatrix}\longleftrightarrow\begin{pmatrix}
0 \\ p(d) \\ p(cb)
\end{pmatrix}$$

Hence, the string complex associated to $w$ is 
$$P_{\omega}^\bullet: \ \xymatrix{\cdots\ar[r] & 0\ar[r] & P_1\ar[r]^(0.3){\partial_{\omega}^{-1}}&P_2\oplus P_2\oplus P_1\ar[r]^(0.7){\partial_{\omega}^0}&P_3\ar[r]&0\ar[r]&\cdots}$$

We observe that there is a ``graphic way'' to construct this complex, as follows. We can represent the generalized string $\omega=a^{-1}\cdot c\cdot d \cdot (cb)^{-1}$ as
$$
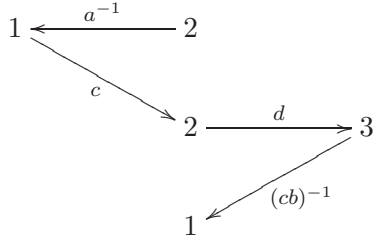
\xymatrix{1\ar[rrd]_c&&2\ar[ll]_{a^{-1}}&&\\
                               &&2\ar[rr]^d        &&3\ar[dll]^{(cb)^{-1}}\\
                               &&1                     &&             }
$$
In the figure above, we put the vertices $s(w_1)=2$, $t(w_1)=1$, $t(w_2)=2$, $t(w_3)=3$ and $t(w_4)=1$ sequentially, such that we go to the left if $w_i$ is an inverse path and we move to the right if $w_i$ is a direct path. By replacing the vertex $i$ by the projective $P_i$ and writing the corresponding morphisms induced between these projective modules ($p(w_i): P_{s(w_i)}\longrightarrow P_{t(w_i)}$ if $w_i\in\textbf{Pa}_{\geq 1}$ or $p(w_i^{-1}): P_{t(w_i)}\longrightarrow P_{s(w_i)}$ if $w_i^{-1}\in\textbf{Pa}_{\geq 1}$), we obtain

$$
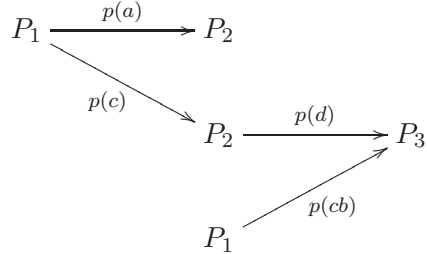
\xymatrix{P_1\ar[rr]^{p(a)}\ar[rrd]_{p(c)}&&P_2&&\\
	&&P_2\ar[rr]^{p(d)}        &&P_3\\
	&&P_1\ar[urr]_{p(cb)}                     &&             }
$$

This diagram represents the complex (we form direct sums by columns)

$$P_{\omega}^\bullet: \ \xymatrix{\cdots\ar[r] & 0\ar[r] & P_1\ar[r]^(0.3){\partial_{\omega}^{-1}}&P_2\oplus P_2\oplus P_1\ar[r]^(0.7){\partial_{\omega}^0}&P_3\ar[r]&0\ar[r]&\cdots}$$
\end{example}

\section{Periodic string complexes}\label{Sec:PSC}
In this section we will give a necessary and sufficient condition for a string complex to have infinite minimal projective resolution when the algebra is string. These complexes are called \textit{periodic string complexes.} We give a generalization of the cyclic sets $Q_c$, $\overline{GSt}_c$, $\overline{GSt}^c$ in \cite{Be-Me}, in order to characterize string complexes with infinite minimal projective resolution.

Let us begin by defining a condition of minimality on paths in the kernel of a morphism $p(w)$. Notice that this condition is automatic for the cases of gentle and SAG algebras, in which the minimal generators of kernels are arrows.

\begin{definition}
	Let $w, u\in\emph{\textbf{Pa}}_{>0}$ with $t(w)=s(u)$ and such that $wu=0$. We say that $w$ is \textit{minimal for $u$}  if no proper subpath $w'$ of $w$ with $t(w')=s(u)$ verifies $w'u=0$.
\end{definition}
Thus, $a^*$ in the notation of the Lemma \ref{Lem-St} satisfies this minimality condition, that is, $a^*$ is minimal for $w$.

Let us define a new set, which generalizes the sets of cyclic arrows $Q_c$ in \cite{Be-Me} and $Q_c^*$ in \cite{FGR}. Let $\textbf{Pa}_c$ be the set of paths $w\in\textbf{Pa}_{>0}$ for which there exist paths $w_m, w_{m-1}, \dots, w_1\in \textbf{Pa}_{>0}$ such that $t(w_m)=s(w)$, $t(w_i)=s(w_{i+1})$ for $i=1,\dots, m-1$, $s(w_1)=t(w_j)$ for some $1\leq j\leq m+1$, where $w_{m+1}=w$, with $w_iw_{i+1}=w_jw_1=0$. In addition we require that $w_i$ is minimal for $w_{i+1}$ and $w_j$ is minimal for $w_1$.

It is clear that $Q_c\subseteq Q_c^*\subseteq \textbf{Pa}_c$, and, in the case of gentle algebras (resp. SAG algebras), we have $Q_c=\textbf{Pa}_c$ (resp. $Q_c^*=\textbf{Pa}_c$). The elements of $\textbf{Pa}_c$ are called \textit{cyclic paths}. 

\begin{example}\label{Ex:Str:2}
Let $(Q, I)$ be the bound quiver
	$$ 
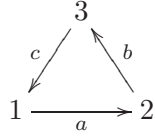
\xymatrix@C=4.5mm{              & 3\ar[dl]_c&                &\\
		1\ar[rr]_a&                & 2\ar[ul]_b&}
	$$
	with $I=\langle ca, abc\rangle$. Then $A=kQ/I$ is a string algebra over $k$, which is neither gentle nor SAG algebra. In this case we have that $Q_c$ and $Q_c^*$ are empty sets but $\emph{\textbf{Pa}}_c=\{a,bc,c,ab\}$.  
	
	Now, consider the generalized string $\omega=w_1=bc$. Then, $l(\omega)=1>0$, $\mu(\omega)=0$ and $\exists a\in\emph{\textbf{Pa}}_c$ such that $a\cdot \omega=a\cdot bc\in GSt$. The complex $P_{\omega}^\bullet$ is 
	$$P_{\omega}^\bullet: \ \xymatrix{\cdots\ar[r]&0\ar[r]&P_2\ar[r]^{p(bc)}&P_1\ar[r]&0\ar[r]&\cdots}$$         
	
	Since $abc=0$ and $ca=0$, then $\ker p(bc)=Aa$, $\ker p(a)=Ac$, $\ker p(c)=Aab$ and $\ker p(ab)=Ac$. Thus, the (minimal) projective resolution $P_{\beta(P_{\omega}^\bullet)^\bullet}^\bullet$ is
	
	$$
	\xymatrix{\cdots\ar[r]&P_1\ar[r]^{p(ab)}&P_3\ar[r]^{p(c)}&P_1\ar[r]^{p(ab)}&P_3\ar[r]^{p(c)}&P_1\ar[r]^{p(a)}&P_2\ar[r]^{p(bc)}&P_1\ar[r]&0}
	$$
	
	Hence $P_{\beta(P_{\omega}^\bullet)^\bullet}^\bullet\notin K^b(\text{pro} \ A)$.

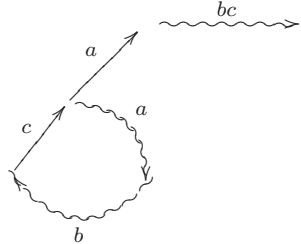
   $$
   \xymatrix{  &            & \ar@{~>}[rr]^{bc} &  &\\
                & \ar[ru]^a \ar@{~>}@/^{4mm}/[rd]^a  &                         &   & \\
                 \ar@[]+<0.4cm,0.2cm>; [ur]+<-0.01cm,-0.3mm>^c&    & \ar@{~>}@/^{5mm}/[]+<0.05mm, 0.2mm>;[ll] +<0.4cm,0.1cm>^b     &   &  \\
                            &    &                           &   &}      
   $$                            
\end{example}
In this example we observe that $a$ is minimal for $bc$ and when we calculate the projective resolution of $P_{\omega}^\bullet$ we get the cycle $\{c, ab\}$, which is a subset of $\textbf{Pa}_c$. The consequence of this is that $P_{\beta(P_{\omega}^\bullet)^\bullet}^\bullet\notin K^b(\text{pro} \ A)$. This motivates one of the conditions in the following generalization of the special sets.

We denote by $\overline{GSt}_{cp}$ the set of generalized strings $\omega=w_1\cdot w_2\cdots w_n$ of positive length such that $\mu(\omega)=0$, there exists $w\in\textbf{Pa}_c$  with $w\cdot\omega\in GSt$  or $w\in\ker p(w_{l}^{-1})\cap \ker p(w_{l+1})$ for some even index $l> 0$ with $\mu_{\omega}(l)=0$,  where $w$ is minimal for $w_1$ or minimal for $w_{l}^{-1}$ and $w_{l+1}$, respectively. 

Also, we denote by $\overline{GSt}^{cp}$ the set of generalized strings $\omega=w_1\cdot w_2\cdots w_n$ of positive length such that $\mu(\omega)=\mu_{\omega}(n)$ and there exists $w\in \textbf{Pa}_c$ such that $\omega\cdot w^{-1}\in GSt$.

Thus, we have that these sets reduce to the corresponding one for the case of string almost gentle algebras in \cite{FGR} (resp. gentle algebras \cite{Be-Me}).

Using this new special sets we can generalize the characterization of periodic string complexes for SAG algebras given in Theorem 31, part 2 in \cite{FGR}. We will divide this characterization into several lemmas here below.

\begin{lemma}\label{Lemma7-n=1-converse}
	Let $\omega=w_1$ be a generalized string. If $\omega\in \overline{GSt}_{cp}$ or $\omega\in \overline{GSt}^{cp}$, then $P_{\beta(P_{\omega}^\bullet)^\bullet}^\bullet\notin K^b(\text{pro} \ A)$.
\end{lemma}
\begin{proof} We will consider only the case $\omega\in \overline{GSt}_{cp}$. The other case is similar.

	If $\omega=w_1\in \overline{GSt}_{cp}$, then there exists $w\in\textbf{Pa}_c$ such that $w\cdot\omega\in GSt$ with $w$ minimal for $w_1$ and $\mu(\omega)=0$.
	
	Since $\mu(\omega)=0$, then $w_1\in\textbf{Pa}_{>0}$ and the condition $w\cdot \omega\in GSt$ implies $ww_1=0$. Now, since $A=k(Q,I)$ is a string algebra, there are at most two arrows $a$ and $b$ ending at vertex $s(w_1)$.
	
	$$
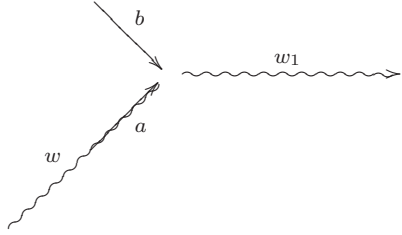
\xymatrix{            &\ar[rd]^{b}&                             &&& \\
		                           &                &\ar@{~>}[rrr]^{w_1}&&& \\
	                               &\ar[ru]_a   &                                      \\
   \ar@{~}[rruu]^(0.4)w&                &                             &&& }$$
   
   We know that it cannot occur that $aw_1\neq 0$ and $bw_1\neq 0$. Suppose, without loss of generality, that $bw_1=0$ and $w$ is a subpath of $\hat{a}a$. (Recall that $\hat{a}$ is unique for string algebras as a consequence of Lemma \ref{lem-left-completion}, and so $a^*=w$. Also, it is possible that $w=a$). Therefore, by Lemma \ref{Lem-St}, it follows that $\ker p(w_1)=Aw \oplus Ab$. (If $bw_1\neq 0$, then $aw_1=0$, $w=a$ and $\ker p(w_1)=Aa$, as in the gentle case. However, we consider here the more general case $\ker p(w_1)=Aw\oplus Ab$).
   
   We have $\mu_{\omega}(0)=0$, $\mu_{\omega}(1)=1$. Then the complex $P_{\omega}^\bullet$ is
   
   $$P_{\omega}^\bullet: \ \xymatrix{\cdots\ar[r]&0\ar[r]&P_{s(w_1)}\ar[r]^{p(w_1)}&P_{t(w_1)}\ar[r]&0\ar[r]&\cdots}$$
   
   In order to determine $P_{\beta(P_{\omega}^\bullet)^\bullet}^\bullet$, the projective resolution of $P_{\omega}^\bullet$, we first note that, since $$p(w): P_{s(w)}\longrightarrow Aw \quad \text{and} \quad p(b): P_{s(b)}\longrightarrow Ab$$
   are projective covers, then $\ker p(w_1)=Aw\oplus Ab$ can be covered with $P^{-1}=P_{s(w)}\oplus P_{s(b)}$ and the epimorphism 
   $$
   \begin{array}{rcl}
   \partial^{-1}: P_{s(w)}\oplus P_{s(b)} & \longrightarrow & Aw\oplus Ab\\
                                        (u,v)           & \mapsto             & (uw,vb)
   \end{array}
   $$
   where $\partial^{-1}=\begin{pmatrix}
   p(w)  & 0 \\ 0  &  p(b)
   \end{pmatrix}$. Besides, $\ker \partial^{-1}=\ker p(w)\oplus \ker p(b)$.
   Since $w\in\textbf{Pa}_c$, by Lemma \ref{Lem-St} we have $\ker p(w)=Au_m\oplus Ac$, where $c$ is an arrow such that $t(c)=s(w)$ and $u_m$ is a cyclic path minimal for $w$. Therefore $\ker\partial^{-1}=Au_m\oplus Ac\oplus\ker p(b)$. Furthermore, $\ker p(b)$ is at most two-dimensional, i.e., $\ker p(b)=Ab_1^*\oplus Ab_2$, for some arrows $b_1, b_2$ ending at $s(b)$ and $b_1^*$ is the smallest subpath of $\widehat{b_1}b_1$ such that $b_1^*b=0$ and $b_2b=0$ (recall that this is the general case in Lemma \ref{Lem-St}).
    
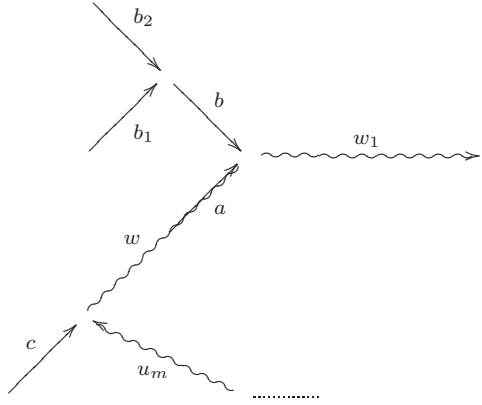
   $$\xymatrix{&       \ar[rd]^{b_2}     &               &                                                            &&&\\
  &                                                 	&\ar[rd]^{b}&                                                           &&& \\
   &                      	\ar[ru]_{b_1}     &                &\ar@{~>}[rrr]^{w_1}                              &&& \\
  &                                               	&\ar[ru]_a   &                                                                     \\
   &	\ar@{~}[rruu]^(0.4)w                  &                &                                                             &&& \\
                        \ar[ru]^c	&               & & \ar@{~>}[llu]^{u_m}                              &\ar@{.}[l]&&}$$
   Thus, $\ker p(b)$ can be covered by $P_{s(b_1^*)}\oplus P_{s(b_2)}$ and the corresponding epimorphisms. Hence, we can cover $\ker\partial^{-1}=Au_m\oplus Ac\oplus Ab_1^*\oplus Ab_2$ with $P^{-2}=P_{s(u_m)}\oplus P_{s(c)}\oplus P_{s(b_1^*)}\oplus P_{s(b_2)}$ and the epimorphism 
   $$\partial^{-2}=\begin{pmatrix}
   p(u_m) & 0      & 0            & 0 \\
   0          & p(c) & 0            & 0 \\
   0          & 0     & p(b_1^*) & 0 \\
   0          & 0     & 0            & p(b_2)
   \end{pmatrix}$$
   where $\ker \partial^{-2}=\ker p(u_m)\oplus\ker p(c)\oplus\ker p(b_1^*)\oplus\ker p(b_2)$. Again, since $w\in\textbf{Pa}_c$,  there is a cyclic path $u_{m-1}$, minimal for $u_m$, such that $u_{m-1}u_m=0$ and thus $Au_{m-1}$ is a direct summand of $\ker \partial^{-2}$. Hence $\ker \partial^{-2}$ can be covered by some projective $P^{-3}$ which has $P_{s(u_{m-1})}$ as a direct summand. 
   
   If we continue in this fashion, by the definition of the cyclic set of paths $\textbf{Pa}_c$, in finitely many steps we will obtain again $P_{s(u_j)}$, for some $1\leq j\leq m+1$, as a direct summand in the projective resolution $P_{\beta(P_{\omega  }^\bullet)^\bullet}^\bullet$ and, since $Q_1$ is finite, eventually the other summands will disappear.  We can repeat the process over the subset $\{u_1, u_2, \dots, u_j\}$ of the cyclic set $\textbf{Pa}_c$. This shows that $$P_{\beta(P_{\omega}^\bullet)^\bullet}^\bullet\notin K^b(\text{pro} \ A).$$
\end{proof}

\begin{lemma}\label{Lemma7-n=2-converse}
Let $\omega=w_1\cdot w_2$ be a generalized string.	If $\omega\in \overline{GSt}_{cp}$ or $\omega\in \overline{GSt}^{cp}$, then $P_{\beta(P_{\omega}^\bullet)^\bullet}^\bullet\notin K^b(\text{pro} \ A).$
\end{lemma}

\begin{proof}
	As before, we only consider the case $\omega\in \overline{GSt}_{cp}$.
	If $\omega=w_1\cdot w_2\in \overline{GSt}_{cp}$, then $\mu(\omega)=0$ and there exists $w\in\textbf{Pa}_c$ such that $w\cdot\omega=w\cdot w_1\cdot w_2\in GSt$ where $w$ is minimal for $w_1$.

Now, for $w_2$ we have two possibilities: $w_2\in\textbf{Pa}_{>0}$ or $w_{2}^{-1}\in\textbf{Pa}_{>0}$.\\
	
	\textit{Case 1} : if $w_2\in\textbf{Pa}_{>0}$, then $\mu_{\omega}(0)=0$, $\mu_{\omega}(1)=1$, $\mu_{\omega}(2)=2$ and $w_1w_2\in I$. The complex $P_{\omega}^\bullet$ is 
	
	$$P_{\omega}^\bullet: \ \xymatrix{\cdots\ar[r]&0\ar[r]&P_{s(w_1)}\ar[r]^(0.5){p(w_1)}&P_{t(w_1)}\ar[r]^{p(w_2)}&P_{t(w_2)}\ar[r]&0\ar[r]&\cdots}$$
	As in Lemma \ref{Lemma7-n=1-converse}, $\ker p(w_1)=Aw\oplus Ab$ has an infinite (minimal) projective resolution (because $w\in \textbf{Pa}_c$) and hence
	$$P_{\beta(P_{\omega }^\bullet)^\bullet}^\bullet\notin K^b(\text{pro} \ A).$$
	
	\textit{Case 2}: if $w_{2}^{-1}\in\textbf{Pa}_{>0}$, then $\mu_{\omega}(0)=0$, $\mu_{\omega}(1)=1$, $\mu_{\omega}(2)=0$ and $w_1w_2\in St$. The complex $P_{\omega}^\bullet$ is 
	
	$$P_{\omega}^\bullet: \ \xymatrix{\cdots\ar[r]&0\ar[r]&P_{s(w_1)}\oplus P_{t(w_2)}\ar[r]^(0.6){\partial_{\omega}^0}&P_{t(w_1)}\ar[r]&0\ar[r]&\cdots}$$ where $\partial_{\omega}^0=\begin{pmatrix}
	p(w_1) \\ p(w_2^{-1})
	\end{pmatrix}$.
	
	We first determine $\ker \partial_{\omega}^0$. For this, $\begin{pmatrix}
	u & v
	\end{pmatrix}\in\ker \partial_{\omega}^0$ if and only if $uw_1+vw_{2}^{-1}=0$ with $u\in P_{s(w_1)}$ and $v\in P_{t(w_2)}$.
	
	Since $A=k(Q,I)$ is a string algebra, there are at most two arrows $a$ and $b$ such that $t(a)=t(b)=s(w_1)$ and at most two arrows $c$ and $d$ such that $t(c)=t(d)=t(w_2)=s(w_{2}^{-1})$.
$$
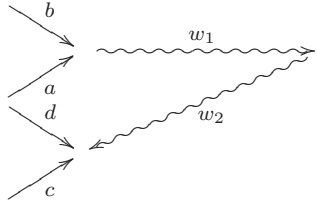
\xymatrix@R=4.5mm{ \ar[rd]^{b}               &                                &&&\\
                                                &\ar@{~>}[rrr]^{w_1} &&&\ar@{~>}[ddlll]^{w_{2}} \\
                 \ar[ru]_a\ar[rd]^{d} &                                &&&\\
                                                &                                &&&\\
                  \ar[ru]_c                &                                &&&}
$$
Notice that it is not possible that $aw_1\neq 0$ and $bw_1\neq 0$ (similarly for $c$ and $d$). Suppose, without loss of generality, that $bw_1=0$, $cw_{2}^{-1}=0$ and $w$ is a subpath of $\hat{a}a$. Since $\omega\in \overline{GSt}_{cp}$, then $ww_1=0$ and therefore $w=a^*$. Thus $\ker p(w_1)=Aw\oplus Ab$. (If $bw_1\neq 0$, then $aw_1=0$, $w=a=a^*$ and $\ker p(w_1)=Aa$). Similarly, in the more general case we have that $\ker p(w_{2}^{-1})$ is of the form $Ac\oplus Ad^*$, where it is possible that $d^*=d$.

Let $\{e_{s(w_1)}, u_1,\dots, u_s\}$ and $\{e_{t(w_2)}, v_1,\dots, v_t\}$ be bases of $P_{s(w_1)}$ and $P_{t(w_2)}$, respectively. Then 
\begin{eqnarray*}
	u & = & \alpha_0e_{s(w_1)}+\alpha_1u_1+\cdots+\alpha_su_s\\
	v & = & \beta_0e_{t(w_2)}+\beta_1v_1+\cdots+\beta_tv_t
\end{eqnarray*}
for some scalars $\alpha_0, \alpha_1, \dots, \alpha_s, \beta_0, \beta_1, \dots, \beta_t$. Now, suppose that $w$ is a subpath of $u_1, \dots, u_k$, that $u_{k+1}, \dots, u_{k+r}$ have $a$ as their last arrow but are smaller paths than $w$ and $u_{k+r+1}, \dots, u_s$ have $b$ as their last arrow (reordering the basis if necessary).

Analogously, suppose that $d^*$ is a subpath of $v_1, \dots, v_q$, that $v_{q+1}, \dots, v_{q+l}$ have $d$ as their last arrow but are smaller paths that $d^*$, and $v_{q+l+1}, \dots, v_t$ have $c$ as their last arrow.

Then, from the condition $uw_1+vw_{2}^{-1}=0$  and the minimality of $w$ and $d^*$ it follows that 
\begin{eqnarray*}
0 & = & (\alpha_0e_{s(w_1)}+\alpha_1u_1+\cdots+\alpha_su_s)w_1+(\beta_0e_{t(w_2)}+\beta_1v_1+\cdots+\beta_tv_t)w_{2}^{-1} \\
   & = & \alpha_0w_1+\alpha_{k+1}u_{k+1}w_1+\cdots \alpha_{k+r}u_{k+r}w_1 \\ 
   &    & + \beta_0w_{2}^{-1}+\beta_{q+1}v_{q+1}w_{2}^{-1}+\cdots+\beta_{q+l}v_{q+l}w_{2}^{-1}
\end{eqnarray*}

Since $w_1w_2\in St$, then $w_1, u_{k+1}w_1,\dots, u_{k+r}w_1, w_{2}^{-1}, v_{q+1}w_{2}^{-1},\dots, v_{q+l}w_{2}^{-1}$ are linearly independent, and hence

$$\alpha_0=\alpha_{k+1},\cdots=\alpha_{k+r}=\beta_0=\beta_{q+1}=\cdots=\beta_{q+l}=0$$
Thus, $$u=\alpha_1u_1+\cdots+\alpha_ku_k+\alpha_{k+r+1}u_{k+r+1}+\cdots+\alpha_su_s\in Aw\oplus Ab$$ and
$$v=\beta_1v_1+\cdots+\beta_qv_q+\beta_{q+l+1}v_{q+l+1}+\cdots+\beta_tv_t\in Ad^*\oplus Ac$$
that is, $$\ker\partial_{\omega}^0=Aw\oplus Ab\oplus Ad^*\oplus Ac$$
We can cover this kernel with $P^{-1}=P_{s(w)}\oplus P_{s(b)}\oplus P_{s(d^*)}\oplus P_{s(c)}$ and the epimorphism $\partial^{-1}:P^{-1}\longrightarrow \ker\partial_{\omega}^0$, where
$$\partial^{-1}=\begin{pmatrix}
p(w) & 0     & 0         & 0\\
0     & p(b) & 0         & 0\\
0     & 0     & p(d^*)  & 0\\
0     & 0     & 0         & p(c)       
\end{pmatrix}$$
Notice that $\ker\partial^{-1}=\ker p(w)\oplus\ker p(b)\oplus \ker p(d^*)\oplus \ker p(c)$. Now, since $w\in\textbf{Pa}_c$, then, as in Lemma \ref{Lemma7-n=1-converse} we have that $\ker p(w)=Au_m\oplus Af$, where $u_m$ is a cyclic path minimal for $w$ and $f$ is an arrow such that $t(f)=s(w)$. In addition, each of $\ker p(b)$, $\ker p(d^*)$ and $\ker p(c)$ is at most two-dimensional as a $k$-vector space and each one can be covered by some projectives $P_{r_b}\oplus P_{l_b}$, $P_{r_{d^*}}\oplus P_{l_{d^*}}$ and $P_{r_c}\oplus P_{l_c}$, respectively. Thus, $\ker \partial^{-1}$ can be covered by $P^{-2}=P_{s(u_m)}\oplus P_{s(f)}\oplus P_{r_b}\oplus P_{l_b}\oplus P_{r_{d^*}}\oplus P_{l_{d^*}}\oplus P_{r_c}\oplus P_{l_c}$ and the epimorphism 
$$\partial^{-2}=\begin{pmatrix}
p(u_m) & 0        & \cdots & 0 \\
0          & *        & \cdots & 0 \\ 
 \vdots & \vdots &\ddots &  \vdots  \\
 0         &  0    &\cdots     & *
\end{pmatrix}$$
where $\ker \partial^{-2}=\ker p(u_m)\oplus M$, and $M$ is an $A$-module. Now, since $w\in \textbf{Pa}_c$, we have that $\ker p(u_m)=Au_{m-1}\oplus Ag$, where $u_{m-1}$ is a cyclic path minimal for $u_m$ and $g$ is an arrow ending at $s(u_m)$. Hence, we can cover $\ker\partial^{-2}$ by some projective $P^{-3}$ which has $P_{s(u_{m-1})}$ as a direct summand. If we continue the reasoning in the same way, by the definition of $\textbf{Pa}_c$, in finitely many steps we will get again $P_{s(u_j)}$, for some $1\leq j\leq m+1$, as a direct summand in the projective resolution $P_{\beta(P_{\omega}^\bullet)^\bullet}^\bullet$ and the other summands will eventually vanish because $Q_1$ is finite. We can repeat indefinitely the same process over the subset $\{u_1, u_2, \dots, u_j\}$ of the cyclic set of paths $\textbf{Pa}_c$. The conclusion is that $$P_{\beta(P_{\omega}^\bullet)^\bullet}^\bullet\notin K^b(\text{pro} \ A).$$
\end{proof}

\begin{lemma}\label{lemma7-n-converse}
	Let $\omega=w_1\cdot w_2\cdots w_n$ be a generalized string. If $\omega\in \overline{GSt}_{cp}$ or $\omega\in \overline{GSt}^{cp}$, then $P_{\beta(P_{\omega}^\bullet)^\bullet}^\bullet\notin K^b(\text{pro} \ A).$
\end{lemma}
\begin{proof}
	Let $\omega=w_1\cdot w_2\cdots w_n$ be a generalized string. Here we will consider the case $\omega\in \overline{GSt}_{cp}$. The case $\omega\in \overline{GSt}^{cp}$ is similar. 
	
	If $\omega\in \overline{GSt}_{cp}$, then $\mu(\omega)=0$ and there exists $w\in\textbf{Pa}_c$ such that $w\cdot \omega=w\cdot w_1\cdot w_2\cdots w_n\in GSt$ or such that $w\in\ker p(w_{l}^{-1})\cap\ker p(w_{l+1})$ for some even index $l> 0$ for which $\mu_{\omega}(l)=0$, where $w$ is minimal for $w_1$ or $w$ is minimal for $w_{l}^{-1}$ and $w_{l+1}$, respectively

	Also, we will consider the case where $w\cdot \omega\in GSt$. The case in which $w\in\ker p(w_{l}^{-1})\cap\ker p(w_{l+1})$ is completely analogous. Since $\mu(\omega)=0$, it follows that $w_1\in\textbf{Pa}_{>0}$ and hence $ww_1=0$.
	
	Since $A=k(Q, I)$ is a string algebra, there are at most two arrows $a$ and $b$ such that $t(a)=t(b)=s(w_1)$. As before, we consider the general situation in which $bw_1=0$ and $w=a^*$. Thus $\ker p(w_1)=Aw\oplus Ab$.
	
	We must consider two cases:\\
	
	\textit{Case 1:} $w_2\in\textbf{Pa}_{>0}$. In this case we have $\mu_{\omega}(0)=0$, $\mu_{\omega}(1)=1$, $\mu_{\omega}(2)=2$ and $w_1w_2\in I$. We represent this situation as follows:
	
	$$
	
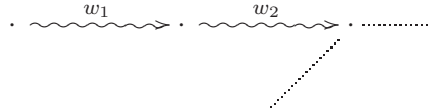
\xymatrix{\cdot\ar@{~>}[rr]^{w_1}&&\cdot\ar@{~>}[rr]^{w_2}&&\cdot\ar@{.}[dl]\ar@{.}[r]&\\
		&&                                    &&}
	$$
	The possibilities for some of the next values of $\mu_{\omega}$ are:
	
	$$\mu_{\omega}(3)=\left\{\begin{array}{l}
	1 \\
	3
	\end{array}\right. ,\quad \mu_{\omega}(4)=\left\{\begin{array}{l} 0\\2\\4\end{array}\right.,\quad \mu_{\omega}(5)=\left\{\begin{array}{l} 1\\3\\5\end{array}\right.$$
	Thus, in the first two places of the complex $P_{\omega}^\bullet$ we have 
	$$P_{\omega}^0=P_{s(w_1)}\oplus0\oplus0\oplus0\oplus \delta(\mu_{\omega}(4),0)P_{t(w_4)}\oplus\cdots$$
	and
	$$P_{\omega}^1=0\oplus P_{t(w_1)}\oplus0\oplus \delta(\mu_{\omega}(3),1)P_{t(w_3)}\oplus 0 \oplus \delta(\mu_{\omega}(5),1)P_{t(w_5)}\oplus\cdots$$
	The complex $P_{\omega}^\bullet$ is of the form
	$$P_{\omega}^\bullet: \ \xymatrix{\cdots\ar[r]&0\ar[r]&P_{\omega}^0\ar[r]^{\partial_{\omega}^0}&P_{\omega}^1\ar[r]^{\partial_{\omega}^1}&P_{\omega}^2\ar[r]&\cdots}$$ 
	where $$\partial_{\omega}^0 = \begin{pmatrix} 
	0         & p(w_1)   &   0       &  0         & 0 &\cdots  & 0          \\
	0         & 0           &   0       &  0         & 0 &\cdots  & 0          \\
	0         & 0           &   0       &  0         & 0 &\cdots  & 0          \\
	0         & 0           &   0       &  0         & 0 &\cdots  & 0          \\ 
	0         & 0           &   0       &  \ast     & 0 &\cdots & 0          \\
	\vdots & \vdots   & \vdots & \vdots   & 0 &\ddots & \vdots  \\
	0         & 0          &  0        & 0          & 0 & \cdots & 0
	\end{pmatrix}
	$$
	and the morphism $\ast: \delta(\mu_{\omega}(4), 0)P_{t(w_4)}\longrightarrow \delta(\mu_{\omega}(3),1)$ is $p(w_4^{-1})$ or zero. Also, it is possible that other entries in the matrix are nonzero, but the important fact is that they are not in the second column.
	
	Now, $(u_0, u_1, \dots , u_n)\in \ker \partial_{\omega}^0$ if and only if $u_0w_1=0$ together with other equations. This shows that $u_0\in\ker p(w_1)=Aw\oplus Ab$ and hence $Aw\oplus Ab$ is a direct summand of $\ker \partial_{\omega}^0$. As in Lemma \ref{Lemma7-n=1-converse}, since $w\in\textbf{Pa}_c$, it follows that $Aw$ has an infinite (minimal) projective resolution, whence the same is true for $\ker\partial_{\omega}^0$. We conclude that  $$P_{\beta(P_{\omega}^\bullet)^\bullet}^\bullet\notin K^b(\text{pro} \ A).$$
	
	\textit{Case 2:} $w_{2}^{-1}\in\textbf{Pa}_{>0}$. In this case we have $\mu_{\omega}(0)=0$, $\mu_{\omega}(1)=1$, $\mu_{\omega}(2)=0$ and $w_1w_2\in St$. Additionally, since $\mu(\omega)=0$, then $w_3\in\textbf{Pa}_{>0}$, that is $\mu_{\omega}(3)=1$. The following diagram represents the situation.
	
	$$
	
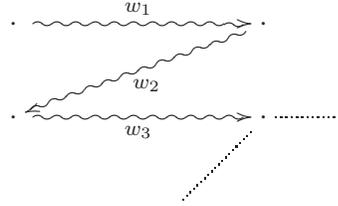
\xymatrix{\cdot\ar@{~>}[rrr]^{w_1}&&&\cdot\ar@{~>}[dlll]^{w_2}&&\\
		\cdot\ar@{~>}[rrr]_{w_3}&&& \cdot\ar@{.}[dl]\ar@{.}[r]                  &&\\
		&&&                                                 &&}
	$$
	The possibilities for some of the next values of $\mu_\omega$ are
	$$\mu_{\omega}(4)=\left\{\begin{array}{l}
	0 \\
	2
	\end{array}\right. ,\quad \mu_{\omega}(5)=\left\{\begin{array}{l} 1\\3\end{array}\right.$$
	Therefore, for $P_{\omega}^0$ and $P_{\omega}^1$ we have
	$$P_{\omega}^0=P_{s(w_1)}\oplus0\oplus P_{t(w_2)}\oplus0\oplus \delta(\mu_{\omega}(4),0)P_{t(w_4)}\oplus0\oplus\cdots$$
	and
	$$P_{\omega}^1=0\oplus P_{t(w_1)}\oplus0\oplus P_{t(w_3)}\oplus 0 \oplus \delta(\mu_{\omega}(5),1)P_{t(w_5)}\oplus 0\oplus\cdots$$
	The complex $P_{\omega}^\bullet$ is of the form
	$$P_{\omega}^\bullet: \ \xymatrix{\cdots\ar[r]&0\ar[r]&P_{\omega}^0\ar[r]^{\partial_{\omega}^0}&P_{\omega}^1\ar[r]^{\partial_{\omega}^1}&P_{\omega}^2\ar[r]&\cdots}$$ 
	where $$\partial_{\omega}^0 = \begin{pmatrix} 
	0         & p(w_1)           &   0       &  0         &\cdots  & 0          \\
	0         & 0                   &   0       &  0         &\cdots  & 0          \\
	0         & p(w_2^{-1})   &   0       & p(w_3)   &\cdots  & 0          \\
	0         & 0                   &   0       &  0         &\cdots  & 0          \\ 
	\vdots & \vdots           & \vdots & \vdots   & \ddots & \vdots  \\
	0         & 0                   &  0        & 0          &  \cdots & 0
	\end{pmatrix}
	$$
	and possibly other entries of the matrix are nonzero but they are not in the second column.
	Now, $(u_0, u_1, u_2, \dots, u_n)\in \ker\partial_{\omega}^0$ if and only if $u_0w_1+u_2w_{2}^{-1}=0$, $u_2w_3+\delta(\mu_{\omega}(4),0)u_4w_{4}^{-1}=0$ and other equations. This implies, by the similar arguments as in Lemma \ref{Lemma7-n=2-converse}, case 2, that $u_0\in Aw\oplus Ab$ and $u_2\in \ker p(w_2^{-1})\cap\ker p(w_3)=\left(Ad^*\oplus Ac\right)\cap \ker p(w_3)$, where $b, c, d\in Q_1$ are such that $t(b)=s(w_1)$ and $t(c)=t(d)=s(w_{2}^{-1})=t(w_2)$. The important facts are that $Aw$ is a direct summand of $\ker\partial_{\omega}^0$ and $w\in\textbf{Pa}_c$. Thus, as in Lemmas \ref{Lemma7-n=1-converse} and \ref{Lemma7-n=2-converse} we conclude that $$P_{\beta(P_{\omega}^\bullet)^\bullet}^\bullet\notin K^b(\text{pro} \ A).$$
\end{proof}
Now, we state and prove the converse of the previous results
\begin{lemma}\label{Lemma7-n=1-direct}
	Let $\omega=w_1$ be a generalized string. If $P_{\beta(P_{\omega}^\bullet)^\bullet}^\bullet\notin K^b(\text{pro} \ A)$, then either $\omega\in \overline{GSt}_{cp}$ or $\omega\in \overline{GSt}^{cp}$.
\end{lemma}
\begin{proof}
	Let $\omega=w_1\in GSt$. Suppose first that $w_1\in\textbf{Pa}_{>0}$, so $\mu_{\omega}(0)=0$ and $\mu_{\omega}(1)=1$. Then $\mu(\omega)=0$ and the complex $P_{\omega}^\bullet$ is 
	$$P_{\omega}^\bullet: \ \xymatrix{\cdots\ar[r]&0\ar[r]&P_{s(w_1)}\ar[r]^{p(w_1)}&P_{t(w_1)}\ar[r]&0\ar[r]&\cdots}$$
	The hypothesis $P_{\beta(P_{\omega}^\bullet)^\bullet}^\bullet\notin K^b(\text{pro} \ A)$ implies that $\ker p(w_1)\neq 0$. Besides, since $A$ is a string algebra, the possibilities for this kernel are: $\ker p(w_1)=Aa^*$ or $\ker p(w_1)=Aa^*\oplus Ab$, where $a, b$ are arrows ending at the vertex $s(w_1)$. 
	
	$$
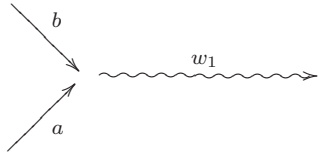
\xymatrix{            &\ar[rd]^{b}&                             &&& \\
		&                &\ar@{~>}[rrr]^{w_1}&&& \\
		&\ar[ru]_a   &                                      }
	$$
	
	We will consider the general case $\ker p(w_1)=Aa^*\oplus Ab$ (it is possible that $a^*=a$). By the hypothesis, $Aa^*\oplus Ab$ has infinite minimal projective resolution, which means that the same is true for either $Aa^*$ or $Ab$. Suppose, without loss of generality, that the minimal projective resolution of $Aa^*$ is not bounded. The projective cover of $Aa^*$ is given by $p(a^*): P_{s(a^*)}\longrightarrow Aa^*$. Again, the hypothesis allows us to choose a path $a_m^*$ such that $a_m$ is an arrow with $t(a_m)=s(a^*)$, $a_m^*a^*=0$ and $\ker p(a_m)$ has infinite minimal projective resolution. 
	
	Analogously, there is a path $a_{m-1}^*$ where $a_{m-1}$ is an arrow with $t(a_{m-1})=s(a_m^*)$, $a_{m-1}^*a_m^*=0$ and $\ker p(a_{m-1}^*)$ has infinite minimal projective resolution. If we continue the process in the same way, since $\textbf{Pa}_{>0}$ is finite, we see that there exist paths $w=a^*, w_m=a_m^*, w_{m-1}=a_{m-1}^* \dots, w_1=a_1^*$ such that $t(w_m)=s(w)$, $t(w_i)=s(w_{i+1})$ for $i=1,\dots, m-1$, $s(w_1)=t(w_j)$ for some $1\leq j\leq m+1$, where $w_{m+1}=w$, and  $w_iw_{i+1}=w_jw_1=0$. By our construction, it is clear that $w_i$ is minimal for $w_{i+1}$ and $w_j$ is minimal for $w_1$.  Since $ww_1=0$, we have shown that there exists $w\in\textbf{Pa}_c$ such that $w\cdot \omega\in GSt$ and $w$ is minimal for $w_1$. Therefore $\omega\in \overline{GSt}_{cp}$.
	
	Finally, if $w_{1}^{-1}\in\textbf{Pa}_{>0}$, by dual arguments it can be shown that $\omega\in\overline{GSt}^{cp}$.
\end{proof}

\begin{lemma}\label{Lemma7-n=2-direct}
	Let $\omega=w_1\cdot w_2\in GSt$. If $P_{\beta(P_{\omega}^\bullet)^\bullet}^\bullet\notin K^b(\text{pro} \ A)$, then either $\omega\in\overline{GSt}_{cp}$ or $\omega\in\overline{GSt}^{cp}$.
\end{lemma}
\begin{proof}
	Let $\omega=w_1\cdot w_2$ be a generalized string. Suppose that $w_1\in\textbf{Pa}_{>0}$. We will show that $\omega\in\overline{GSt}_{cp}$. If $w_{1}^{-1}\in\textbf{Pa}_{>0}$, by dual arguments it can be shown that $\omega\in\overline{GSt}^{cp}$. Now, we have two cases to consider: $w_2\in\textbf{Pa}_{>0}$ or $w_{2}^{-1}\in\textbf{Pa}_{>0}$.\\
	
	\textit{Case 1}: If $w_2\in\textbf{Pa}_{>0}$, $\mu_{\omega}(0)=0$, $\mu_{\omega}(1)=1$, $\mu_{\omega}(2)=2$ and $w_1w_2\in I$. It follows that $\mu(\omega)=0$ and the complex $P_{\omega}^\bullet$ is 
	
	$$P_{\omega}^\bullet: \ \xymatrix{\cdots\ar[r]&0\ar[r]&P_{s(w_1)}\ar[r]^(0.5){p(w_1)}&P_{t(w_1)}\ar[r]^{p(w_2)}&P_{t(w_2)}\ar[r]&0\ar[r]&\cdots}$$
	
	We know that $\ker p(w_1)=Aa^*\oplus Ab$ for some arrows $a, b$ such that $t(a)=t(b)=s(w_1)$. The hypothesis $P_{\beta(P_{\omega}^\bullet)^\bullet}^\bullet\notin K^b(\text{pro} \ A)$ implies that this kernel has an infinite minimal projective resolution. Thus, we are in the same position as in Lemma \ref{Lemma7-n=1-direct}. Therefore, using the same arguments we have that there exist $w\in \textbf{Pa}_c$ such that $w\cdot \omega\in GSt$, where $w$ is minimal for $w_1$. Hence $\omega\in\overline{GSt}_{cp}$. \\
	
	\textit{Case 2}: If $w_{2}^{-1}\in\textbf{Pa}_{>0}$, $\mu_{\omega}(0)=0$, $\mu_{\omega}(1)=1$, $\mu_{\omega}(2)=0$ and $w_1w_2\in St$. Thus, $\mu(\omega)=0$ and the complex $P_{\omega}^\bullet$ is 
	
	$$P_{\omega}^\bullet: \ \xymatrix{\cdots\ar[r]&0\ar[r]&P_{s(w_1)}\oplus P_{t(w_2)}\ar[r]^(0.6){\partial_{\omega}^0}&P_{t(w_1)}\ar[r]&0\ar[r]&\cdots}$$ where $\partial_{\omega}^0=\begin{pmatrix}
	p(w_1) \\ p(w_2^{-1})
	\end{pmatrix}$.
	Since the hypothesis implies that $\ker\partial_{\omega}^0\neq 0$ and $A$ is a string algebra, the general form of this kernel is $\ker\partial_{\omega}^0=\ker p(w_1)\oplus \ker p(w_{2}^{-1})=Aa^*\oplus Ab\oplus Ad^*\oplus Ac$ for some arrows $a, b, c, d$ such that $t(a)=t(b)=s(w_1)$, $t(c)=t(d)=s(w_{2}^{-1})=t(w_2)$ (see Lemma \ref{Lemma7-n=2-converse}), and at least one direct summand is different from zero.
	$$
	
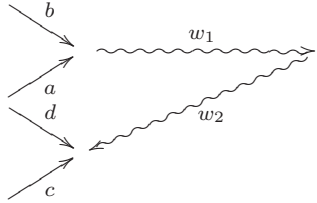
\xymatrix@R=4.5mm{ \ar[rd]^{b}               &                                &&&\\
		&\ar@{~>}[rrr]^{w_1} &&&\ar@{~>}[ddlll]^{w_{2}} \\
		\ar[ru]_a\ar[rd]^{d} &                                &&&\\
		&                                &&&\\
		\ar[ru]_c                &                                &&&}
	$$
	
	Since $P_{\beta(P_{\omega}^\bullet)^\bullet}^\bullet\notin K^b(\text{pro} \ A)$, then either $\ker p(w_1)$ or $\ker p(w_{2}^{-1})$ has an infinite minimal projective resolution.\\
	
	If $\ker p(w_1)=Aa^*\oplus Ab$ has infinite minimal projective resolution, then the same is true for either $Aa^*$ or $Ab$ and we proceed as in Lemma \ref{Lemma7-n=1-direct} to show that there exists $w\in\textbf{Pa}_c$ such that $w\cdot \omega\in GSt$ with $w$ minimal for $w_1$. Therefore $w\in \overline{GSt}_{cp}$.\\
	
	If $\ker p(w_{2}^{-1})=Ad^*\oplus Ac$ has infinite minimal projective resolution, the same holds for either $Ad^*$ or $Ac$. In this case we consider $\omega^{-1}=w_{2}^{-1}\cdot w_{1}^{-1}$ instead of $\omega$ and use the same arguments as in Lemma \ref{Lemma7-n=1-direct} to show that there is a path $w\in\textbf{Pa}_c$ such that $ww_{2}^{-1}=0$ (that is, $w\cdot \omega^{-1}\in GSt$), where $w$ is minimal for $w_{2}^{-1}$. This means that $\omega^{-1}\in\overline{GSt}_{cp}$, that is $\omega\in\overline{GSt}_{cp}$. (Here, we are using the fact that $\omega\cong_s \omega^{-1}$ and $P_{\omega}^{\bullet}\cong P_{\omega^{-1}}^{\bullet}$ in $D^b(A)$).
\end{proof}

Now, we use Lemmas \ref{Lemma7-n=1-direct} and \ref{Lemma7-n=2-direct} to prove the same result for the general case. 

\begin{lemma}\label{Lemma7-n-direct}
	Let $\omega=w_1\cdot w_2\cdots w_n$ be a generalized string. If $P_{\beta(P_{\omega}^\bullet)^\bullet}^\bullet\notin K^b(\text{pro} \ A)$, then either $\omega\in\overline{GSt}_{cp}$ or $\omega\in\overline{GSt}^{cp}$.
\end{lemma}
\begin{proof}
	As in the previous lemma, we will show that if $w_1\in\textbf{Pa}_{>0}$, then $\omega\in\overline{GSt}_{cp}$. By dual arguments it can be shown that if $w_{1}^{-1}\in\textbf{Pa}_{>0}$, then $\omega\in\overline{GSt}^{cp}$.
		
	Let us begin by assuming that $\mu(\omega)=0$, that is, we suppose that $\omega$ has as number of direct paths greater than or equal to the number of inverse paths.  We are allowed to do so since, if $\omega$ has more inverse than direct paths, we consider $\omega^{-1}$ instead of $\omega$.
 	If $w_1\in \textbf{Pa}_{>0}$, as $\mu(\omega)=0$, then $\mu_{\omega}(1)=1$ and the possibilities for some of the next values of $\mu_{\omega}$ are 
 	
 	$$\mu_{\omega}(2)=\left\{\begin{array}{l}
 	0 \\
 	2
 	\end{array}\right. ,\quad \mu_{\omega}(3)=\left\{\begin{array}{l} 1\\3\end{array}\right. ,\quad \mu_{\omega}(4)=\left\{\begin{array}{l} 0\\2\\4 \end{array}\right., \quad \mu_{\omega}(5)=\left\{\begin{array}{l} 1\\3\\5 \end{array}\right.
 	$$
 	Thus, if $\mu_{\omega}(l)=0$ with $l>0$, then $l$ is an even number greater than or equal to two and less that $n$. Also, it is possible that $\mu_{\omega}(n)=0$ when $n$ is even. It follows that the complex $P_{\omega}^\bullet$ is 
 	$$P_{\omega}^\bullet: \ \xymatrix{\cdots\ar[r]&0\ar[r]&P_{\omega}^0\ar[r]^{\partial_{\omega}^0}&P_{\omega}^1\ar[r]^{\partial_{\omega}^1}&P_{\omega}^2\ar[r]&\cdots}$$ 
 	where the general form of $\ker\partial_{\omega}^0$ is either one of the following
 	\begin{enumerate}
 		\item\label{ker:1} $\ker\partial_{\omega}^0\cong \ker p(w_1)\oplus\cdots\oplus \ker p(w_{l}^{-1})\cap \ker p(w_{l+1})\oplus\cdots$, whenever $n$ is odd and $\mu_{\omega}(l)=0$ for some $2\leq l\leq n-1$.
 		\item\label{ker:2} $\ker\partial_{\omega}^0\cong \ker p(w_1)\oplus\cdots\oplus \ker p(w_{l}^{-1})\cap \ker p(w_{l+1})\oplus\cdots\oplus \ker p(w_{n}^{-1})$, whenever $n$ is even, there is an even index $2\leq l\leq n-2$ with $\mu_{\omega}(l)=0$ and $\mu_{\omega}(n)=0$.
 		\item\label{ker:3} $\ker\partial_{\omega}^0\cong \ker p(w_1)\oplus\cdots\oplus \ker p(w_{l}^{-1})\cap \ker p(w_{l+1})\oplus\cdots$, whenever $n$ is even, there is an even index $2\leq l\leq n-2$ with $\mu_{\omega}(l)=0$ and $\mu_{\omega}(n)>0$.
 	
 	\end{enumerate}
 Now, since $P_{\beta(P_{\omega}^\bullet)^\bullet}^\bullet\notin K^b(\text{pro} \ A)$, then $\ker\partial_{\omega}^0$ has an infinite minimal projective resolution. If $\ker\partial_{\omega}^0$ is as in (\ref{ker:1}), then at least one direct summand also has an infinite minimal projective resolution. If it is the case for $\ker p(w_1)$, we proceed as in Lemma \ref{Lemma7-n=1-direct}, to show that there is a path $w\in\textbf{Pa}_c$ such that $w\cdot \omega\in GSt$, where $w$ is minimal for $w_1$, and hence $\omega\in \overline{GSt}_{cp}$. 
 
 If the minimal projective resolution of $\ker p(w_{l}^{-1})\cap\ker p(w_{l+1})$ is not bounded for some $2\leq l\leq n-1$, the same arguments allow us to show that there exist $w\in\textbf{Pa}_c$ such that $w\in\ker p(w_{l}^{-1})\cap\ker p(w_{l+1})$, where $w$ is minimal for both $w_{l}^{-1}$ and $w_{l+1}$. Therefore, the conclusion follows.
 
 If $\ker\partial_{\omega}^0$ is as in (\ref{ker:3}), the reasoning is the same. Finally, if $\ker\partial_{\omega}^0$ is as in (\ref{ker:2}) and $\ker p(w_{n}^{-1})$ has infinite minimal projective resolution, we consider ${\omega}^{-1}$ instead of $\omega$ and, as in the proof of Lemma \ref{Lemma7-n=2-direct}, case 2, it can be shown that $\omega\cong_s {\omega}^{-1}\in\overline{GSt}_{cp}$.
\end{proof}

Now, if we put Lemmas \ref{Lemma7-n=1-converse} to \ref{Lemma7-n-direct} together, we get the following theorem, which provides a characterization for a string complex to be periodic in the case of a string algebra.

\begin{theorem}\label{Lemma7-String}
	Let $A=k(Q, I)$ be a string algebra. If $\omega=w_1\cdot w_2\cdots w_n$ is a generalized string, then $P_{\beta(P_{\omega}^\bullet)^\bullet}^\bullet\notin K^b(\text{pro} \ A)$ if and only if either $\omega\in\overline{GSt}_{cp}$ or $\omega\in\overline{GSt}^{cp}$.
\end{theorem}

\section{Applications}\label{Sec:App}
Two important applications of the Theorem \ref{Lemma7-String}, which we detail below, are related to the global dimension and the indecomposables of the derived category of a string algebra.

\subsection{About of the global dimension of a string algebra}
 Let $A$ be a finite dimensional $k$-algebra. Recall that given an $A$-module $M$, the \textit{projective dimension} of $M$, denoted by $pd \ M$ is the smallest integer $d$ such that there exists a projective resolution of the form
$$\xymatrix{0\ar[r]&P_d\ar[r]&P_{d-1}\ar[r]&\cdots\ar[r]&P_1\ar[r]&P_0\ar[r]&M}.$$
If no resolution exists, then we say that $M$ has infinite projective dimension. The \textit{global dimension of $A$}, denoted by $\text{gl.dim} A$ is defined as the supremum of the projective dimensions of all $A$-modules, that is, $$\text{gl.dim} A:=\text{sup}\{pd \ M \mid M\in A-\text{mod}\}.$$

Theorem \ref{Lemma7-String} is very important because it deals with the structure of a string algebra. That is, it establishes when a string algebra has infinite global dimension. We state this in the following result.

\begin{theorem}\label{Th:gl.dim-St}
	Let $A=k(Q, I)$ be a string algebra. If $\overline{GSt}_{cp}\neq\emptyset$ or $\overline{GSt}^{cp}\neq\emptyset$, then $\text{gl.dim} A=\infty$.
\end{theorem}

\begin{proof}
	Suppose that $\overline{GSt}_{cp}\neq\emptyset$. The case $\overline{GSt}^{cp}\neq\emptyset$ is similar.
If $\omega=w_1\cdot w_2\cdots w_n\in \overline{GSt}_{cp}$, then $\mu(\omega)=0$ and there exists $w\in \textbf{Pa}_c$ such that $w\cdot \omega=w\cdot w_1\cdot w_2\cdots w_n\in GSt$ or  $w\in\ker p(w_{l}^{-1})\cap\ker p(w_{l+1})$ for some even index $l> 0$ with $\mu_{\omega}(l)=0$.
	
	Again, we will consider the case where $w\cdot \omega\in GSt$. The case in which $w\in\ker p(w_{l}^{-1})\cap\ker p(w_{l+1})$ is completely analogous. If $w\cdot \omega\in GSt$ and $\mu(\omega)=0$, necessarily $w_1\in\textbf{Pa}_{>0}$ and hence $ww_1=0$. According to Lemma \ref{Lem-St}, the general structure of $\ker p(w_1)$ is $\ker p(w_1)=Aw\oplus Ab$, where $b$ is an arrow such that $t(b)=s(w_1)$.
	Following the proof of Lemma \ref{lemma7-n-converse}, since $w\in\textbf{Pa}_c$, we get that $Aw$ has infinite minimal projective resolution and this implies that $\text{gl.dim} A=\infty$.
\end{proof}

\begin{example}\label{Example-Qc*} Let $Q$ be the bound quiver
	
	$$ 
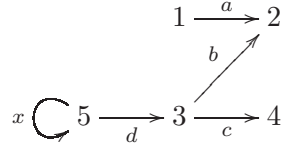
\xymatrix{&1\ar[r]^a&2\\
		5\ar@(ul,dl)[]_{x} \ar[r]_d &3\ar[r]_c\ar[ur]^b&4}
	$$	
with $I=\langle db, dc, x^2, xd\rangle$. Then $A=kQ/I$ is a string algebra over $k$. In this case we have that $\emph{\textbf{Pa}}_c=\{x,d,c,b\}$.  
	
	Now, consider the generalized string $\omega=w_1\cdot w_2\cdot w_3=a\cdot b^{-1}\cdot c$. Then, $l(\omega)=3>0$, $\mu(\omega)=0$ and $\exists d\in \emph{\textbf{Pa}}_c$ such that $d\in \ker p(w_{2}^{-1})\cap\ker p(w_3)=\ker p(b)\cap \ker p(c)$, with $\mu_{\omega}(2)=0$. That is, $\omega\in \overline{GSt}_{cp}$.
	
	According to Definition \ref{def-st-comp}, the complex $P_{\omega}^\bullet$ is 
	$$P_{\omega}^\bullet: \ \xymatrix{\cdots\ar[r]&0\ar[r]&P_1\oplus P_3\ar[r]^{\partial_{\omega}^0}&P_2\oplus P_4\ar[r]&0\ar[r]&\cdots}$$ 
	
	where $\partial_w^0=\begin{pmatrix}
	p(a)  & 0\\
	p(b)  & p(c)
	\end{pmatrix}$        
	
	Now, $\begin{pmatrix}u & v\end{pmatrix}\in\ker \partial_{\omega}^0$ if and only if $ua+vb=0$ and $vc=0$. Since $u\in P_1$ and $v\in P_3$, then $u=\alpha_0e_1$ and $v=\beta_0e_3+\beta_1d$ for some scalars $\alpha_0, \beta_0, \beta_1\in k$.
	
	From $ua+vb=0$, it follows that $\alpha_0a+\beta_0b=0$, and hence $\alpha_0=\beta_0=0$. Notice that the condition $vc=0$ is already satisfied. Therefore, we have that $u=0$ and $v=\beta_1d\in Ad=\ker p(b)\cap\ker p(c)$ and thus $$\ker \partial_{\omega}^0\cong \ker p(b)\cap\ker p(c)=Ad$$ (in general, we have for this case that $\ker\partial_{\omega}^0=\ker p(w_1)\oplus\ker p(w_2^{-1})\cap\ker p(w_3)$).	
	
	The projective cover of this kernel is $p(d): P_5\longrightarrow Ad$ and $\ker p(d)=Ax$. Since $x^2=0$, it is clear that the minimal projective resolution $P_{\beta(P_{\omega}^\bullet)^\bullet}^\bullet$ of $P_{\omega}^\bullet$ is
	
	$$\xymatrix@C=4mm@R=4.7mm{\cdots\ar[rr]&&P_5\ar[rr]^{p(x)}\ar@{->>}[dr]&                              &P_5\ar[rr]^{p(d)}\ar@{->>}[dr]&                           &P_{\omega}^0\ar[rr]^{\partial_{\omega}^0}&&P_{\omega}^1\ar[rr]&&0\ar[r]&\cdots\\
		&&                                                             &Ax\ar@{^(->}[ru]&                                                    & Ad\ar@{^(->}[ru]&                                       &&                       && }
	$$
	that is, $P_{\beta(P_{\omega}^\bullet)^\bullet}^\bullet\notin K^b(\text{pro} \ A)$.       
	
	Notice also that, since the generalized string $\omega=w_1\cdot w_2\cdot w_3=a\cdot b^{-1}\cdot c \in\overline{GSt}_{cp}$, then, according to Theorem \ref{Th:gl.dim-St}, we have that $\text{gl.dim} A=\infty$.         
\end{example}

\subsection{About the indecomposables in $D^b(A)$}
 For string algebras with the property that every arrow belongs to a unique maximal path, as in example \ref{Ex:Str:1}(2), the functor $F: \mathfrak{p}(A)\longrightarrow s(\mathcal{Y}, k)$ in \cite{FGR} assigns to every  generalized strings $\omega$ an indecomposable object in $D^b(A)$. In fact, this is a consequence of Theorem \ref{Lemma7-String} and Proposition \ref{prop:1}. Precisely,

\begin{theorem}
	Let $A=k(Q,I)$ be a string algebra with the property that every arrow belongs to a unique maximal path. Then
	\begin{eqnarray*}
		\text{ind}_0  D^b(A) &\supseteq&\{T^i(P_w^\bullet) \mid w\in GSt, i\in\mathbb{Z}\}\dot{\cup}\\
		&  &\{T^i(\beta(P_{\omega}^\bullet)^\bullet) \mid \omega\in \overline{GSt}_{cp}, i\in\mathbb{Z}\}\dot{\cup}\\
		&  & \{T^i(\beta(P_{\omega}^\bullet)^\bullet) \mid \omega\in \overline{GSt}^{cp}\setminus \overline{GSt}_{cp}, i\in\mathbb{Z}\}.
	\end{eqnarray*}
	where $T$ is the translation functor.
\end{theorem}
\begin{proof}
The proof follows the same reasoning as the Main Theorem in \cite{FGR} (Theorem 37), because the important fact is the \textit{unique maximal path property}. In fact, Lemmas 24 and 35 in \cite{FGR} are still valid in this case, which are essential in the proof and are based on the uniqueness of maximal paths.
\end{proof}

\subsection*{Acknowledgments}
This paper was prepared for the proceedings of the meeting {\it Geometry in Algebra and Algebra in Geometry}, GAAG-V, held in Medell\'{i}n in 2019. The first author would like to thank the organizers for the invitation to speak and to Colciencias (Beca Doctorado Nacional Colciencias, Convocatoria 647 de 2014). The authors are also grateful to CODI (Universidad de Antioquia, U de A). The GAAG-V and the third author, were partially supported by CODI, University of Antioquia, project 2017-15756 Stable Limit Linear Series on Curves. 


\end{document}